\documentclass[11pt]{article}
\usepackage[margin=1.3in]{geometry}
\usepackage{comment}
\usepackage{amsthm}
\usepackage{mathrsfs}
\usepackage{fullpage}
\usepackage{setspace}
\usepackage{url}
\usepackage{tikz}
\usetikzlibrary{fit, positioning}
\usepackage{mathtools}
\usetikzlibrary{er}
\usepackage{amsmath,amsfonts,amssymb}
\usepackage{verbatim,graphics}
\usepackage{microtype}
\usepackage{xcolor}

\usepackage{bbm}
\usepackage{graphicx}
\usepackage{tikz}
\usetikzlibrary{calc}
\usepackage{tabularx}
\usepackage{subcaption}
\usepackage[shortlabels]{enumitem}
\usepackage{comment}
\newcommand{\ep}{\varepsilon}

\def\NN{{\mathbb N}}

\def\RR{{\mathbb R}}

\let\temp\phi
\let\phi\varphi
\let\varphi\temp

\def\conv{\operatorname {conv}}

\def\Vert{\operatorname {Vert}}

\def\dist{\operatorname {dist}}

\usepackage[backend=biber,style=alphabetic]{biblatex}
\usepackage{hyperref}
\theoremstyle{definition}
\newtheorem{definition}{Definition}[section]

\theoremstyle{plain}
\newtheorem{thm}[definition]{Theorem}
\newtheorem{obs}[definition]{Observation}
\newtheorem{lemma}[definition]{Lemma}

\newtheorem{cor}[definition]{Corollary}
\newtheorem*{cor*}{Corollary}

\begin{document}

\title{Sharp Threshold for Cliques in Random 0/1 Polytope Graphs}

\author{Catherine Babecki \thanks{Departments of Computing and Mathematical Sciences and of Mathematics, California
Institute of Technology, Pasadena, CA 91125 
Email:  {{\tt cbabecki@caltech.edu}}. Research is supported by  Simons Grant MPS-SIM-00623756.
}  \and Tycho Elling \thanks{Department of Mathematics, University of California, Irvine. Email: {\tt telling@uci.edu}}\and Asaf Ferber \thanks{Department of Mathematics, University of California, Irvine. 
Email:  {{\tt asaff@uci.edu}}. Research is supported by NSF Career DMS-2146406 and an Air Force grant FA9550-23-1-0298.} }

\maketitle

\begin{abstract}
    We study graph-theoretic properties of random $0/1$ polytopes. Specifically, let $Q_p^n \subseteq \{0,1\}^n$ be a random subset where each point is included independently with probability $p$, and consider the graph $G_p$ of the polytope $\conv(Q_p^n)$.
We provide a short and combinatorial proof that $p = 2^{-n/2}$ is a threshold for the edge density of $G_p$, a result originally due to Kaibel and Remshagen. 
We next resolve an open question from their paper by showing that for $p \leq 2^{-n/2 - o(1)}$, $G_p$ exhibits strong edge expansion.
In particular, we prove that, with high probability, every vertex has degree $(1 - o(1))|Q_p^n|$. 
Lastly, we determine the threshold for $G_p$ being a clique, strengthening a result of Bondarenko and Brodskiy. 
We show that with high probability, if
\( p \geq  2^{-\delta n + o(1)} \), then $G_p$ is not a clique, and if \( p \leq 2^{-\delta n - o(1)}\), then \(G_p\) is a clique, where \(\delta \approx 0.8295\).
Our approach combines a combinatorial characterization of edges in graphs arising from polytopes with the Kim-Vu polynomial concentration inequality.
\end{abstract}

\section{Introduction}

Polytopes with 0-1 vertices, known as $0/1$ polytopes, are of great practical interest due to their role in many combinatorial optimization problems, such as the matroid base polytope, the perfect matching polytope, the cut polytope, and the stable set polytope.
We refer the reader to \cite{SchrijverBooks} for an extensive reference on the connections between polyhedra and combinatorial optimization.

Graphs of polytopes first gained attention for their role in the simplex algorithm in the context of linear programming.
More recently, graphs of polytopes have been used to sample uniformly at random from a set of combinatorial objects by running a random walk on the associated graph.
This method works well if the walk mixes quickly enough. 
\emph{Edge expansion} of a graph is enough to guarantee rapid mixing as a consequence of Cheeger's Inequality. We will return shortly to a precise definition of edge expansion, but loosely speaking, a graph has strong edge expansion if every subset of vertices is well-connected to the rest of the graph. 

The Mihail-Vazirani Conjecture asserts that every 0/1 polytope graph has edge expansion at least 1 \cite{MihailVaziraniConj}, which is sufficient for combinatorial algorithms that depend on rapid mixing.
There has been a significant body of work towards this conjecture, which remains open in full generality. The conjecture has been proven for certain classes of 0/1 polytopes; in particular we note the breakthrough work \cite{logconcavepolynomials} which proves the conjecture for any matroid base polytope.
We refer to \cite[Section 3]{LerouxRademacherExpansion} for a more thorough literature review of 0/1 polytopes and expansion. Kaibel and Remshagen \cite{KaibelRemshagenDensity} asked whether the Mihail-Vazirani Conjecture holds for \emph{random} 0/1 polytopes, where the vertices of the polytope are selected from among $\{0,1\}^n$ independently with probability $p$. The work of \cite{LerouxRademacherExpansion} shows that a weaker version of the conjecture holds on average: the edge expansion of a random $n$-dimensional 0/1 polytope
is at least $1/12n$ with high probability. Their proofs are geometric in nature and the result is agnostic to the probability with which vertices are included in the random polytope.  We also note that upcoming work of the third author, Krivelevich, Sales, and Samotij shows that with high probability, the edge expansion of a random 0/1-polytope is bounded below by an absolute constant.

In contrast, our result in Theorem~\ref{thm: edge-expansion} shows that if $p < 2^{-n/2}$, the edge expansion is $\Omega(|V|)$. 
That is, while we cannot say something about every probability $p$, we have nearly maximal edge expansion when $p$ is small. 
Our methods are purely combinatorial, and provide a simpler proof that $p = 2^{-n/2}$ is the threshold for the edge density of the graph of a random 0/1 polytope, first proven in \cite{KaibelRemshagenDensity}. Further, our technique establishes the threshold for when the graph of a random 0/1 polytope is a clique, improving on bounds due to Bondarenko and Brodskiy \cite{BondarenkoBrodskiiCliques}. Precise statements will follow some mathematical preliminaries.

\subsection{Preliminaries}

We first recall some of the basics of polyhedra and refer the reader to \cite{GrunbaumBook,ZieglerBook} for a more complete introduction. 
A \emph{polytope} $P \subset \RR^n$ is the convex hull of a finite set of points in $\RR^n$. That is, for any finite set $S = \{x_1,\ldots, x_m\} \subset \RR^n$, the convex hull
\[
\conv(S) := \left\{ \sum_{i=1}^m \lambda_i x_i \,\middle|\, \lambda_i \geq 0,\ \sum_{i =1}^m \lambda_i = 1 \right\}
\]
is a polytope. A \emph{supporting hyperplane} of a polytope $P$ is a hyperplane $H$ for which $P \cap H$ is nonempty, and $P$ is contained in a closed halfspace defined by $H$. A \emph{face} of a polytope $P$ is the intersection of $P$ with a supporting hyperplane. A $k$-dimensional face is called a $k$-\emph{face}. In particular, $0$-faces are called \emph{vertices}, and $1$-faces are called \emph{edges}. We denote the set of vertices of $P$, i.e., its $0$-dimensional faces, by \( \Vert(P) \). In particular, \( \Vert(P) \) is the minimal set \( Q \subset \mathbb{R}^n \) such that \( P = \conv(Q) \).
The \emph{graph} (or \emph{$1$-skeleton}) of a polytope $P$ is the undirected graph $G_P$ whose vertex set consists of the vertices of $P$, and whose edge set consists of the $1$-faces of $P$. We illustrate some of these concepts in Figure~\ref{fig:poly pics}.

A \emph{$0/1$-polytope} in $\RR^n$ is the convex hull of a subset of the Boolean hypercube $\{0,1\}^n$, i.e., a polytope whose vertices all have coordinates in $\{0,1\}$. 
In this work, we study \emph{random $0/1$-polytopes} and the graphs associated with them.
Let $\mathcal{Q}^n := \{0,1\}^n$ denote the $n$-dimensional Boolean hypercube.
For a probability parameter $0 \leq p \leq 1$, define $Q^n_p \subseteq \mathcal{Q}^n$ as a random subset obtained by including each point $x \in \mathcal{Q}^n$ independently with probability $p$.
The \emph{random $0/1$ polytope} is then 
\(
P_p := \mathrm{conv}(Q_p^n),
\)
and we define the associated graph $G_p$ as the $1$-skeleton of $P_p$. That is, $G_p$ has vertex set $Q_p^n$, and its edge set is the set of $1$-faces of $P_p$. For two points $x,y\in \RR^n$, we use $[x,y]$ to denote the line segment between $x$ and $y$.

\begin{figure}[h]
    \centering
\begin{tabular}{cc}
    \subfloat[]{
   \begin{tikzpicture}[baseline={($ (current bounding box.west) - (0,1ex) $)},scale = .8]
   \tikzstyle{bk}=[circle, fill =black ,inner sep= 2 pt,draw]
     \fill[black!10] (0, -.5) to (0, 2.5) to (2.5, 2.5) to (2.5,-.5) to (0, -.5) ;
\node (a) at (.4,2.3) {\tiny{$01$}};
\node (b) at (2.3, 2.3)  {\tiny{$11$}};
\node (c) at (2.3, -.3) {\tiny{$10$}};
\node (d) at (.4, -.3)  {\tiny{$00$}};
\node (H) at (-.4,1)  {\textcolor{red}{H}};
 \draw[thick, fill =cyan!30] (2,0) to (0,0) to (0,2) to (2,2) to (2,0);
 \node (P) at (1,1)  {\textcolor{blue}{P}};
  \draw[line width =.7 mm,red, <->] (0, -.5) to (0, 2.5);
 \node (1) at (0,2) [bk] {};
\node (2) at (2,2) [bk] {};
\node (3) at (2,0) [bk] {};
\node (4) at (0,0) [bk] {};
    \end{tikzpicture}} \hspace{1 in}
    & 
        \subfloat[]{
   \begin{tikzpicture}[baseline={($ (current bounding box.west) - (0,1ex) $)},scale = .8 ]
   \tikzstyle{bk}=[circle, fill = white ,inner sep= 1 pt,draw]
 \node (1) at (0,2) [bk] {10};
\node (2) at (2,2) [bk] {11};
\node (3) at (2,0) [bk] {01};
\node (4) at (0,0) [bk] {00};
 \draw[thick] (3) to (4) to (1) to (2) to (3);
    \end{tikzpicture}}
      \end{tabular}
    \caption{ (A) shows the square, a 0/1 polytope in $\RR^2.$ Its vertex set is $ \mathcal{Q}^2 $, and the supporting hyperplane $H$ shows that $[00,01]$ is an edge of $P.$
        The graph or 1-skeleton $G_P$ is shown in (B). For instance, $\{00,11\} \notin E(G_P)$ because there is no hyperplane which contains both $00$ and $11$ and supports $P.$}
    \label{fig:poly pics}
\end{figure}
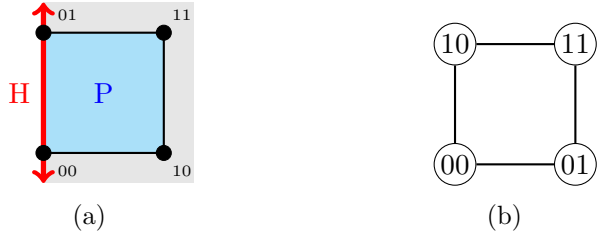

Our goal is to understand the typical combinatorial properties of the graph $G_p$ --  degree distribution, edge-density, and edge expansion -- as the sampling probability $p$ varies with $n$. Given a graph \( G \), we write \( V(G) \) and \( E(G) \) for its sets of vertices and edges, respectively. When $G$ is clear from context, we simply write \( V \) and \( E \). 
The degree of a vertex \( v \in V(G) \) is denoted \( \deg(v)\).

\begin{definition}
 Given a graph $G$, we define its \emph{edge density} as
\[
d(G) := \frac{|E(G)|}{\binom{|V(G)|}{2}}.
\]
\end{definition}
In particular, $0 \leq d(G) \leq 1$, and $d(G) = 1$ if and only if $G$ is a clique.

\begin{definition}
 The \emph{edge expansion} of a graph \( G = (V, E) \) is defined as
\[
\Phi(G) :=  \min_{\substack{\varnothing \subset S \subset V \\ |S| \leq |V|/2 }} \left(\frac{|E(S, V \setminus S)|}{|S|} \right),
\]
where \( E(S, V \setminus S) \) is the set of edges with one endpoint in \( S \) and the other in \( V \setminus S \).   
\end{definition}

\subsection{Main Results}

We begin by providing a very short and simple proof of the following theorem, originally due to Kaibel and Remshagen~\cite{KaibelRemshagenDensity}.

\begin{thm} \label{thm: edge-density}
Let $P_p$ be a random $0/1$ polytope, let $G_p$ be its corresponding graph, and fix $\varepsilon > 0$. Then, with high probability:
\begin{enumerate}
    \item If $p \leq \left(\frac{1 - \varepsilon}{\sqrt{2}}\right)^{n}$, then
    \[
    d(G_p) = 1 - o(1).
    \]
    
    \item If $p \geq \left(\frac{1+\varepsilon}{\sqrt{2}}\right)^{n}$, then
    \[
    d(G_p) = o(1).
    \]
\end{enumerate}
\end{thm}

Next, addressing a problem posed in~\cite{KaibelRemshagenDensity}, we show that below the threshold \( 2^{-n/2} \), the graph $G_p$ exhibits strong edge expansion. 
In fact, we prove a stronger statement: for values of \( p \) below the threshold \( 2^{-n/2} \), the graph is not only dense but nearly regular, with all vertices having degree close to the maximum possible. 

\begin{thm} \label{thm: edge-expansion}
Let \( \varepsilon > 0 \), and suppose \( p \leq \left(\frac{1 - \varepsilon}{\sqrt{2}}\right)^n \). Let \( P_p \) be a random $0/1$ polytope, and let \( G_p \) be its associated graph. Then, with high probability, every vertex \( v \in V(G_p) \) satisfies
\[
\deg(v) = (1 - o(1)) \cdot |V(G)|.
\]
\end{thm}

This immediately implies strong edge expansion. Specifically: 

\begin{cor}
    Let \( \varepsilon > 0 \), and suppose \( p \leq \left(\frac{1 - \varepsilon}{\sqrt{2}}\right)^n \). Let \( P_p \) be a random $0/1$ polytope, and let \( G_p \) be its associated graph. Then, with high probability, 
    \[ \Phi(G) = \Omega(|V|).\]
\end{cor}

Although Theorem~\ref{thm: edge-expansion}  implies part~(1) of Theorem~\ref{thm: edge-density}, we include a complete proof of the latter, as it is concise, self-contained, and illustrative of our approach. Lastly, we determine the threshold for when $G_p$ is a clique. This threshold improves on both bounds established in~\cite{BondarenkoBrodskiiCliques}.
Throughout this paper, \( H(\delta) = -\delta \log_2 \delta - (1-\delta) \log_2 (1-\delta) \) is the binary entropy function.

\begin{thm} \label{thm: clique threshold}
Fix $\varepsilon > 0$, and let $\delta>1/2$ satisfy $H(\delta)=2\delta -1$ ($\delta\approx 0.8295$).
Let $P_p$ be a random $0/1$ polytope, let $G_p$ be its graph, Then, with high probability:
\begin{enumerate}
    \item If $p \leq \left( (1 - \varepsilon) 2^{-\delta} \right)^n$, then $G_p$ is a clique.
    
    \item If $p \geq \left((1+\ep)2^{-\delta}\right)^n$, then $G_p$ is not a clique.
\end{enumerate}
\end{thm}

\section{Auxiliary Results}

We collect several auxiliary results that will be used throughout the proofs of our main theorems. We use \( \mathcal{P} \) to denote polynomials, and \( \mathbb{P} \) to denote probabilities.

\subsection{Some Concentration Inequalities}

Since our arguments are probabilistic in nature, we make frequent use of the following concentration inequalities.
We begin with the well-known Chernoff bounds for the upper and lower tails of sums of independent indicator random variables (see~\cite{JLR}):

\begin{lemma}[Chernoff Bound]\label{chernoff}
Let \( X_1, \ldots, X_n \) be independent $0/1$-valued random variables, and let \( X = \sum_{i=1}^n X_i \) with \( \mu := \mathbb{E}[X] \). Then for any \( \delta \in (0,1) \),
\[
\mathbb{P}(|X - \mu| > \delta \mu) \leq 2e^{-\delta^2 \mu / 3}.
\]
\end{lemma}

Next, we will make extensive use of a powerful concentration inequality for low-degree polynomials of independent random variables, due to Kim and Vu~\cite{KimVuConcentration}. To state it formally, we first introduce some notation.

Let \( X_1, \ldots, X_m \) be independent random variables, and let \( \mathcal P(X_1, \ldots, X_m) \) be a polynomial of degree \( k \). For any subset \( A \subseteq [m] \), define
\[
\partial_A := \prod_{j \in A} \frac{\partial}{\partial X_j}
\]
to be the partial derivative operator with respect to the variables \( \{X_j : j \in A\} \). Define the expected partial derivatives by
\[
\mathbb{E}_j[\mathcal P] := \max\left\{ \mathbb{E}[\partial_A \mathcal P] : A \subseteq [m],\ |A| = j \right\}.
\]

The following theorem bounds the deviation of a polynomial from its expectation:

\begin{thm}[Kim--Vu~\cite{KimVuConcentration}]\label{lem: concentration inequality}
Let \( \mathcal P(X_1, \ldots, X_m) \) be a polynomial of degree \( k \) in independent, identically distributed random variables \( X_1, \ldots, X_m \). Then there exist some constants $a_k, b_k$ such that for any \( \ell \geq 1 \) we have 
\[
\mathbb{P}\left[ \left| \mathcal P - \mathbb{E}[\mathcal P] \right| \geq a_k\ell^k \cdot \sqrt{ \mathbb{E}_0[\mathcal P] \cdot \mathbb{E}_1[\mathcal P] } \right] 
\leq \exp\left( -\ell + (k - 1)\log m +b_k\right).
\] 
\end{thm}

\subsection{Combinatorial Characterization of Edges}

We begin with a combinatorial criterion for identifying edges (i.e., $1$-faces) in $0/1$-polytopes, which plays a central role throughout our analysis.
The following characterization of edges follows quickly from first principles of polytopes which we illustrate in Figure~\ref{fig:edge condition}.

\begin{obs} \label{def: edges}
Let $P$ be a polytope and let $x, y \in \mathrm{Vert}(P)$. Then, $x$ and $y$ form a $1$-face of $P$ if and only if
\[
\conv(x, y) \cap \conv(\Vert(P) \setminus \{x, y\}) = \varnothing.
\]
\end{obs}

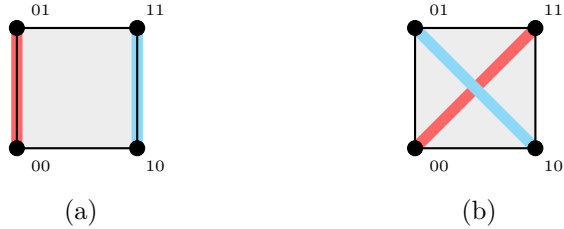
\begin{figure}[h]
    \centering
\begin{tabular}{cc}
    \subfloat[]{
   \begin{tikzpicture}[baseline={($ (current bounding box.west) - (0,1ex) $)},scale = .8]
   \tikzstyle{bk}=[circle, fill =black ,inner sep= 2 pt,draw]
\node (a) at (.4,2.3) {\tiny{$01$}};
\node (b) at (2.3, 2.3)  {\tiny{$11$}};
\node (c) at (2.3, -.3) {\tiny{$10$}};
\node (d) at (.4, -.3)  {\tiny{$00$}};
\fill[black!7] (2,0) to (0,0) to (0,2) to (2,2) to (2,0);
 \draw[line width = 1.5 mm, red!60] (0,0) to (0,2);
  \draw[line width = 1.5 mm, cyan!40] (2,0) to (2,2);
 \draw[thick] (2,0) to (0,0) to (0,2) to (2,2) to (2,0);
 \node (1) at (0,2) [bk] {};
\node (2) at (2,2) [bk] {};
\node (3) at (2,0) [bk] {};
\node (4) at (0,0) [bk] {};
    \end{tikzpicture}} \hspace{1 in}
    & 
        \subfloat[]{
   \begin{tikzpicture}[baseline={($ (current bounding box.west) - (0,1ex) $)},scale = .8 ]
      \tikzstyle{bk}=[circle, fill =black ,inner sep= 2 pt,draw]
\node (a) at (.4,2.3) {\tiny{$01$}};
\node (b) at (2.3, 2.3)  {\tiny{$11$}};
\node (c) at (2.3, -.3) {\tiny{$10$}};
\node (d) at (.4, -.3)  {\tiny{$00$}};
\fill[black!7] (2,0) to (0,0) to (0,2) to (2,2) to (2,0);
 \draw[line width = 1.5 mm, red!60] (0,0) to (2,2);
  \draw[line width = 1.5 mm, cyan!40] (2,0) to (0,2);
 \draw[thick] (2,0) to (0,0) to (0,2) to (2,2) to (2,0);
 \node (1) at (0,2) [bk] {};
\node (2) at (2,2) [bk] {};
\node (3) at (2,0) [bk] {};
\node (4) at (0,0) [bk] {};
    \end{tikzpicture}}
      \end{tabular}
    \caption{ (A) The edge $[00,01]$ is disjoint from the convex hull of the remaining vertices. 
  (B) The non-edge $[00,11] $ intersects the convex hull of the remaining vertices.  In particular, $((0,0) + (1,1))/2 = ((0,1) +(1,0))/2$.}
    \label{fig:edge condition}
\end{figure}

In other words, the segment $[x,y]$ is an edge of $P$ if it is not contained in the convex hull of the remaining vertices of $P$.
On the other hand, if $xy \notin E(G_P)$, then there exists a finite subset $Z \subseteq \Vert(P)\setminus \{x,y\}$ such that
\begin{equation} \label{eq: conv comb condition}
 \alpha x + (1 - \alpha)y = \sum_{z \in Z} \lambda(z) z,
\end{equation}
where $\alpha \in (0,1)$, $\lambda(z) > 0$ for all $z \in Z$, and $\sum_{z \in Z} \lambda(z) = 1$. That is, the point $\alpha x + (1 - \alpha)y$ lies in the convex hull of other vertices of $P$.
We have the following as an immediate consequence using a specific choice of $\lambda(z)$.

\begin{obs} \label{obs: edges}
Let $P$ be a polytope and let $x \neq y \in \Vert(P)$. Suppose there exist an integer $k \in \NN$ and vertices $z_1, \ldots, z_{2k} \in \Vert(P) \setminus \{x, y\}$ such that
\[
\frac{1}{2}x + \frac{1}{2}y = \frac{1}{2k} \sum_{i=1}^{2k} z_i.
\]
Then the segment $[x,y]$ is not an edge of $P$, i.e., $xy \notin E(G_P)$.
\end{obs}

Now, assume that $\Vert(P) \subseteq \mathcal{Q}^n$. 
We recall that for $0/1$ vectors $x$ and $y$, the operations $\land$ and $\lor$ are interpreted coordinate-wise as 
\begin{align*}
&x(i) \land y(i) = 
    \begin{cases}
        1 & \text{ if } x(i) = 1 \text{ and } y(i) = 1 \\
        0 & \text{ else}
    \end{cases} \\
    & x(i) \lor y(i) = 
    \begin{cases}
        1 & \text{ if } x(i) = 1 \text{ or } y(i) = 1  \\
        0 & \text{ else}
    \end{cases}.
\end{align*}
It was observed in \cite{BondarenkoBrodskiiCliques} that the condition (\ref{eq: conv comb condition}) imposes a strong restriction: in such a convex combination, each $z \in Z$ must satisfy the coordinate-wise inequality
\[
x \land y \leq z \leq x \lor y.
\]

Motivated by this, we introduce the following definition. 
\begin{definition}Let $P $ be a $0/1$ polytope, and let $x,y\in \Vert(P)$. 
The set   of \emph{witnesses} is
\[
W_P(x, y) := \left\{ z \in V \setminus \{x, y\} \,\middle|\, x \land y \leq z \leq x \lor y \right\}.
\]
\end{definition}
We rephrase a lemma of \cite{BondarenkoBrodskiiCliques} in this language, which will be key in our analysis.

\begin{lemma}\cite[Lemma 1]{BondarenkoBrodskiiCliques}\label{lem: russian gadget}
Let $P := \conv(Q)$ with $Q \subseteq \mathcal{Q}^n := \{0,1\}^n$. Then,
\begin{enumerate}
    \item Given $x, y \in Q$, if $|W(x,y)| \leq 1$, then the segment $[x, y]$ is an edge of $P$.
    \item If $|W(x,y)| \leq 1$ for all $x, y \in Q$, then the polytope graph $G_P$ is a clique.
\end{enumerate}
\end{lemma}

\begin{proof}
As observed in the discussion preceding the lemma, suppose that
\[
\alpha x + (1 - \alpha)y = \sum_{z \in Z} \lambda(z) z,
\]
for some $\alpha \in (0,1)$, where $\lambda(z) > 0$ for all $z \in Z$, and $\sum_{z \in Z} \lambda(z) = 1$. Then every $z \in Z$ is a witness, that is, $Z \subseteq W(x,y)$.

Moreover, since all points are $0/1$ vectors, the midpoint $(x + y)/2$ cannot itself be a $0/1$ vector unless $x = y$.
Hence, the convex combination in $Z$ must involve at least two distinct points.
Therefore, $|W(x,y)| \geq |Z| \geq  2$.
So if $|W(x,y)| \leq 1$, then no such convex combination exists.
Hence $[x,y]$ must be an edge of $P$, proving part (1).
Part (2) follows immediately by applying part (1) to all pairs $x,y \in Q$.
\end{proof}

\subsection{Typical and Atypical Points}

Our proof strategy involves classifying points in $\mathcal{Q}^n$ according to whether Lemma~\ref{lem: russian gadget} (or a generalization thereof) applies, and whether such points are likely to appear in the random subset ${Q}_p^n$. To formalize this, we introduce the following terminology.

\begin{definition}
Let $\alpha > 0$, and consider any two strings $x, y \in \mathcal{Q}^n$.
\begin{enumerate}
    \item We say that $y \in \{0,1\}^n$ is \emph{$(x, \alpha)$-typical} if the following holds:
    \begin{align*}
        \frac{|x|}{2} - \alpha n \ \leq\ &|x \land y| \ \leq\ \frac{|x|}{2} + \alpha n, \\
        \frac{n - |x|}{2} - \alpha n \ \leq\ &|(x \lor y)^c| \ \leq\ \frac{n - |x|}{2} + \alpha n.
    \end{align*}
    Otherwise, we say that $y$ is \emph{$(x, \alpha)$-atypical}.
    
    \item We say that the pair $(x, y)$ is an \emph{$\alpha$-typical pair} if $y$ is $(x, \alpha)$-typical and $x$ is $(y, \alpha)$-typical. Otherwise, we say that $(x, y)$ is an \emph{$\alpha$-atypical pair}.
\end{enumerate}
\end{definition}

We will make use of the following simple but useful lemma.

\begin{lemma} \label{lem: alpha typical pair}
Let $x, y \in \mathcal{Q}^n$. If $(x, y)$ is an $\alpha$-typical pair, then
\[
\frac{n}{2} - 2\alpha n \leq |x|, |y| \leq \frac{n}{2} + 2\alpha n, \quad \text{and} \quad \frac{n}{4} - 2\alpha n \leq |x \land y| \leq \frac{n}{4} + 2\alpha n.
\]

\end{lemma}

\begin{proof}
Assume that $(x, y)$ is an $\alpha$-typical pair. 
Recall that
\[
|(x \lor y)^c| = n - |x| - |y| + |x \land y|.
\]
Therefore, 
\begin{align*}
\frac{n - |x|}{2} - \alpha n &\leq  |(x \lor y)^c| = n - |x| - |y| + |x \land y| \\
\Rightarrow -\alpha n &\leq \frac{n}{2} - \frac{|x|}{2} - |y| + |x \land y| \\
&\leq \frac{n}{2} - |y| + \alpha n \quad \text{(since $|x \land y| \leq \frac{|x|}{2} + \alpha n$)} \\
\Rightarrow |y| &\leq \frac{n}{2} + 2\alpha n.
\end{align*}

Similar reasoning using the upper bound on $|(x \lor y)^c|$ and the symmetry between $x$ and $y$ shows
\[
\frac{n}{2} - 2\alpha n \leq |x|, |y| \leq \frac{n}{2} + 2\alpha n.
\]
The bounds on $|x \land y|$ follows by substituting in these values.
\end{proof}

In the following lemma, we show that for every \( x \in \mathcal{Q}^n \), the number of \((x, o(1))\)-atypical points is exponentially small in \( n \).

\begin{lemma}[Atypical neighbors are rare]\label{lem: atypical-pairs}
Let \( x \in \mathcal{Q}^n \) and let \( \alpha > 0 \). Then, there exists a constant \( c = c(\alpha) > 0 \) such that the number of points \( y \in \mathcal{Q}^n \) that are \((x, \alpha)\)-atypical is at most \( 2^{(1 - c)n} \).
\end{lemma}

\begin{proof}
Note that if a pair \( (x, y) \) is \(\alpha\)-typical then the Hamming distance between \( x \) and \( y \) is within the interval \([n/2 - 2\alpha n, n/2 + 2\alpha n]\). Hence, \( y \) is \((x, \alpha)\)-atypical if it differs from \( x \) in fewer than \( n/2 - 2\alpha n \) or more than \( n/2 + 2\alpha n \) coordinates.

The number of such points is the number of binary vectors of length \( n \) whose Hamming distance from \( x \) lies outside the interval \( [n/2 - 2\alpha n, n/2 + 2\alpha n] \). This corresponds to the total weight in the tails of the binomial distribution \( \mathrm{Bin}(n, 1/2) \). By standard Chernoff-type bounds, the total number of such vectors is at most
\[
\sum_{|i - n/2| > 2\alpha n} \binom{n}{i} \leq 2^{(1 - c)n}
\]
for some constant \( c = c(\alpha) > 0 \).
\end{proof}

In the following lemma, we show that in a typical instance of \( Q_p^n \), for every fixed \( x \in Q_p^n \), most points \( y \in Q_p^n \) form \((x, o(1))\)-typical pairs.

\begin{lemma} \label{lem: no too many atypical}
Let \( \alpha > 0 \) and let \( 0 \leq p = p(n) \leq 1 \). Then, there exists a constant \( \beta = \beta(\alpha) > 0 \) such that, with high probability, the following holds:
\begin{enumerate}
    \item If \( 0 \leq p \leq 2^{-(1 - \beta)n} \), then for all \( x, y \in Q_p^n \), \( y \) is \((x, \alpha)\)-typical.
    
    \item If \( 2^{-(1 - \beta)n} \leq p \leq 1 \), then for every \( x \in \mathcal{Q}^n \), the number of points \( y \in Q_p^n \) that are \((x, \alpha)\)-atypical is \( o(2^n p) \).
\end{enumerate}
\end{lemma}

\begin{proof}
Fix \( x \in \mathcal{Q}^n \), and let
\[
\mathcal{A}(x) := \left\{ y \in \mathcal{Q}^n \mid y \text{ is } (x, \alpha)\text{-atypical} \right\}.
\]
Define the random variable
\[
Z_x := |\mathcal{A}(x) \cap Q_p^n|,
\]
which satisfies \( Z_x \sim \mathrm{Bin}(|\mathcal{A}(x)|, p) \). By Lemma~\ref{lem: atypical-pairs}, there exists a constant \( c = c(\alpha) > 0 \) such that
\[
|\mathcal{A}(x)| \leq 2^{(1 - c)n}.
\]
Hence,
\(
\mathbb{E}[Z_x] \leq 2^{(1 - c)n} p.
\)

\smallskip

To prove part (1), note that the expected number of pairs \( (x, y) \in Q_p^n \times Q_p^n \) with \( y \) being $(x,\alpha)$-atypical is:
\[
\sum_{x \in \mathcal{Q}^n} \mathbb{E}[Z_x] \cdot \mathbb{P}[x \in Q_p^n] \leq 2^n \cdot 2^{(1 - c)n} p^2 = 2^{(2 - c)n} p^2.
\]
If \( p \leq 2^{-(1 - c/3)n} \), then this expected value is \( o(1) \). By Markov’s inequality, with high probability there are no such pairs in \( Q_p^n \). Setting \( \beta := c/3 \) completes the proof of part (1).

\smallskip

For part (2), suppose now that \( p \geq 2^{-(1 - \beta)n} = 2^{-(1 - c/3)n} \), and fix any \( \varepsilon > 0 \). Note that since \( \varepsilon \cdot 2^n p = \omega(1) \), standard large deviation bounds for binomial random variables show that
\[
\mathbb{P}\left[ Z_x \geq \varepsilon \cdot 2^n p \right] 
\leq \binom{2^{(1 - c)n}}{\varepsilon \cdot 2^n p} \cdot p^{\varepsilon \cdot 2^n p}
\leq \left( \frac{e \cdot 2^{(1 - c)n}}{\varepsilon \cdot 2^n} \right)^{\varepsilon \cdot 2^n p}
= 2^{-\omega(n)}.
\]
Applying a union bound over all \( x \in \mathcal{Q}^n \) (of size \( 2^n \)), we conclude that, with high probability, no vertex \( x \in \mathcal{Q}^n \) has more than \( \varepsilon \cdot 2^n p = o(2^n p) \) atypical neighbors in \( Q_p^n \), completing the proof.
\end{proof}

\section{Proofs}

In this section, we prove all of our main theorems. 

\subsection{Proof of Theorem \ref{thm: edge-density} (edge-density).}

\subsubsection*{Proof of Part (1)}

Let \( \varepsilon, \alpha > 0 \) and \( k \in \mathbb{N} \) such that
\[
0 < 100\alpha < \frac{1}{k} < \frac{\varepsilon}{100},
\]
and suppose \( p \leq \left((1 - \varepsilon)/\sqrt{2}\right)^n \). We will show that, with high probability, the polytope graph \( G_p \) satisfies \( d(G_p) = 1 - o(1) \).
To this end, we establish the following:

\begin{enumerate}[(a)]

    \item The number of \(\alpha\)-atypical pairs in \( \mathcal{Q}^n \) is at most \( 2^{(2 - c)n} \) for some constant \( c = c(\alpha) > 0 \).
    
    \item For any \(\alpha\)-typical pair \( x, y \in \mathcal{Q}^n \), we have
    \[
    \mathbb{P}[xy \notin E(G_p) \mid x, y \in Q_p^n] = o(1).
    \]
    
    \item With high probability, \( |Q_p^n| = (1 - o(1)) \cdot 2^n p \).
    
    \item Let \( X \) denote the number of missing edges in \( G_p \). With high probability,
    \[
    X = o(2^{2n}p^2) = o(|Q_p^n|^2).
    \]

\end{enumerate}
These four statements together imply that \( G_p \) is nearly complete, i.e., \( d(G_p) = 1 - o(1) \), completing the proof of part~(1).

\smallskip

(a) follows directly from Lemma~\ref{lem: atypical-pairs} by summing over all \( x \in \mathcal{Q}^n \).

We next consider (b). Fix an \(\alpha\)-typical pair \( x, y \in \mathcal{Q}^n \). By Lemma~\ref{lem: alpha typical pair}, the number of coordinates on which \( x \) and \( y \) agree is at least \( n/2 - 4\alpha n \). Thus, the number of vertices \( z \in \mathcal{Q}^n \) satisfying
\[
x \land y \leq z \leq x \lor y
\]
is at most \( 2^{n/2 + 4\alpha n} \). A union bound shows that the probability that at least two such vertices \( z,w \in Q_p^n \) exist is at most
\[
2^{n + 8\alpha n} \cdot p^2 = 2^{8\alpha n} \cdot (1 - \varepsilon)^{2n} = o(1),
\]
provided \( \alpha \) is sufficiently small relative to \( \varepsilon \). In particular, if no such distinct \( z,w \) exist, then \( |W(x, y)| \leq 1 \), and by Lemma~\ref{lem: russian gadget}, it follows that \( xy \in E(G_p) \).

(c) follows from a standard application of the Chernoff bound.

For (d), let \( X \) be the number of missing edges in \( G_p \). By linearity of expectation,
\[
\mathbb{E}[X] = \sum_{\substack{(x, y)\ \text{typical}}} \mathbb{P}[xy \notin E(G_p)] + \sum_{\substack{(x, y)\ \text{atypical}}} \mathbb{P}[xy \notin E(G_p)].
\]
Using the identity \( \mathbb{P}[A] = \mathbb{P}[A \mid B]  \mathbb{P}[B] \) with the bounds above, we obtain
\[
\mathbb{P}[xy \notin E(G_p)] = \mathbb{P}[xy \notin E(G_p) \mid x, y \in Q_p^n] \cdot \mathbb{P}[x, y \in Q_p^n] = o(1) \cdot p^2
\]
for \(\alpha\)-typical pairs \( (x, y) \). Since the number of \(\alpha\)-atypical pairs is at most \( 2^{(2 - c)n} \), each of which appears in \( Q_p^n \) with probability \( p^2 \), we conclude:
\[
\mathbb{E}[X] \leq o(1) \cdot 2^{2n}p^2 + 2^{(2 - c)n}p^2 = o(2^{2n}p^2),
\]
as desired. Applying Markov’s inequality completes the proof of part~(1).

\subsubsection*{Proof of Part (2)}

Let \( \varepsilon > 0 \) and suppose \( p \geq \left((1 + \varepsilon)/\sqrt{2}\right)^n \). Our goal is to show that, with high probability, the polytope graph \( G_p \) satisfies \( d(G_p) = o(1) \).

We use the same parameters \( \alpha, k \) as in part~(1). We will show that for every \(\alpha\)-typical pair \( (x, y) \in \mathcal{Q}^n \), the following holds:
\begin{equation} \label{gadget}
    \mathbb{P}\left[\exists\, z_1, \ldots, z_{2k} \in Q_p^n \text{ such that } \frac{1}{2}(x + y) = \frac{1}{2k} \sum_{j=1}^{2k} z_j\right] = 1 - o(1).
\end{equation}

By Lemma~\ref{lem: russian gadget}, the existence of such points \( z_1, \ldots, z_{2k} \in Q_p^n \) implies \( xy \notin E(G_p) \). Thus, Equation~\eqref{gadget} implies
\[
\mathbb{P}[xy \in E(G_p) \mid x, y \in Q_p^n] = o(1).
\]

Let \( Y \) be the number of edges in \( G_p \). Repeating the logic from part~(1), we have:
\[
\mathbb{E}[Y] = \sum_{\substack{(x, y)\ \text{typical}}} \mathbb{P}[xy \in E(G_p)] + \sum_{\substack{(x, y)\ \text{atypical}}} \mathbb{P}[xy \in E(G_p)],
\]
and hence
\[
\mathbb{E}[Y] \leq o(1) \cdot 2^{2n}p^2 + 2^{(2 - c)n}p^2 = o(|Q_p^n|^2),
\]
where we use the fact that \( |Q_p^n| = \Theta(2^n p) \) with high probability. By Markov’s inequality, we conclude that \( Y = o(|Q_p^n|^2) \) with high probability, and hence \( d(G_p) = o(1) \), completing the proof of part~(2), modulo the proof of Equation~\eqref{gadget}.

\subsubsection*{Proof of Equation~\eqref{gadget}}

To prove Equation~\eqref{gadget}, we define a polynomial that counts the number of witness tuples \( (z_1, \ldots, z_{2k}) \subseteq Q_p^n \) such that
\[
\frac{1}{2}(x + y) = \frac{1}{2k} \sum_{j=1}^{2k} z_j.
\]
For each \( v \in \mathcal{Q}^n \), let \( X_v \) be the indicator for the event \( v \in Q_p^n \). Define a polynomial
\[
\mathcal{P}\big((X_v)_{v \in \mathcal{Q}^n}\big) := \sum X_{z_1} \cdots X_{z_{2k}},
\]
where the sum is over all tuples \( (z_1, \ldots, z_{2k}) \in (\mathcal{Q}^n)^{2k} \) (in a fixed order, say lexicographic) satisfying the averaging condition above. Each valid tuple contributes one monomial, and each monomial appears exactly once.

We aim to show that \( \mathbb{E}[\mathcal{P}] = \omega(1) \), and that with high probability, \( \mathcal{P}>0 \). To do so, we apply Theorem~\ref{lem: concentration inequality} (the Kim–Vu polynomial concentration inequality).

\vspace{.1 in}

{\bf Lower bound on \( \mathbb{E}[\mathcal{P}] \).}
For each coordinate \( i \in [n] \), let \( x(i) \), \( y(i) \) denote the \( i \)th bits of \( x \) and \( y \), respectively. If \( x(i) \neq y(i) \), then the \( i \)th coordinate of \( (x + y)/2 \) is \( 1/2 \), so we require that among the \( 2k \) values \( z_j(i) \), exactly \( k \) are 0 and \( k \) are 1.

This constraint must hold independently for each such coordinate. Since \( (x, y) \) is \( \alpha \)-typical, the number of coordinates where \( x(i) \neq y(i) \) lies in the interval \( \left[ n/2 - 2\alpha n,\ n/2 + 2\alpha n \right] \). Therefore, the number of valid tuples is at least
\[
\binom{2k}{k}^{n/2 - 2\alpha n},
\]
and each tuple appears with probability \( p^{2k} \) . Thus we have
\begin{align*}
\mathbb{E}[\mathcal{P}] 
&\geq  \binom{2k}{k}^{n/2 - 2\alpha n} \cdot p^{2k} \\
&\geq   \left( \frac{2^{2k}}{k} \right)^{n/2 - 2\alpha n} \cdot \left( \frac{1 + \varepsilon}{\sqrt{2}} \right)^{2kn} \\
&=   2^{[2k\log_2(1 + \varepsilon) - (1/2 - 2\alpha)\log_2 k - 4\alpha]n}.
\end{align*}
This tends to infinity as \( n \to \infty \), provided \( \alpha \) is sufficiently small relative to \( \varepsilon \), and \( k \) is sufficiently large relative to \( 1/\varepsilon \). Thus, \( \mathbb{E}[\mathcal{P}] = \omega(1) \).

\vspace{.1 in}

{\bf Upper bound on \( \mathbb{E}_1[\mathcal{P}] \).}
To apply Theorem~\ref{lem: concentration inequality}, we bound the first-order partial derivatives. For each \( z \in \mathcal{Q}^n \), \( \frac{\partial \mathcal{P}}{\partial X_z} \) counts the number of monomials in \( \mathcal{P} \) containing \( X_z \), i.e., the number of valid tuples \( (z_1, \ldots, z_{2k}) \) where some \( z_j = z \) and the averaging condition is satisfied.

Fixing such a tuple and removing \( z \), the remaining \( 2k - 1 \) vectors must collectively satisfy a slightly altered version of the same constraint. For each differing coordinate \( i \), we now require exactly \( k - 1 \) of the remaining vectors to take the same value as of \( z(i) \). The number of such configurations is at most
\[
\binom{2k - 1}{k - 1}^{n/2 + 2\alpha n},
\]
and each contributes \( p^{2k - 1} \) to the expectation. Hence,
\begin{align*}
\mathbb{E}_1[\mathcal{P}] 
&\leq \binom{2k - 1}{k - 1}^{n/2  + 2\alpha n} \cdot p^{2k - 1} \\
&= \left( \frac{1}{2} \binom{2k}{k} \right)^{n/2 + 2\alpha n} \cdot p^{2k - 1}.
\end{align*}

Now we express this in terms of \( \mathbb{E}[\mathcal{P}] \):
\begin{align*}
\mathbb{E}_1[\mathcal{P}]
&\leq \left(\frac{1}{2}\right)^{n/2 + 2\alpha n} \cdot \binom{2k}{k}^{4\alpha n} \cdot p\cdot \mathbb{E}[\mathcal{P}] \\
&\leq (1 - \beta)^n \cdot \mathbb{E}[\mathcal{P}]
\end{align*}
for some constant \( \beta = \beta(\varepsilon) > 0 \), since \( \alpha k < 1/100 \) by assumption.

\vspace{.1 in}

{\bf Concentration.}
We now apply Theorem~\ref{lem: concentration inequality}. Since
\[
\sqrt{\mathbb{E}_0[\mathcal{P}] \cdot \mathbb{E}_1[\mathcal{P}]} \leq (1 - \beta)^{n/2} \cdot \mathbb{E}[\mathcal{P}],
\]
taking (say) \( \ell := (1 + \beta)^{n / (5k)} \), we obtain
\begin{align*}
\mathbb{P}[\mathcal{P} = 0] 
&\leq \mathbb{P}\left[ |\mathcal{P} - \mathbb{E}[\mathcal{P}]| \geq \mathbb{E}[\mathcal{P}] \right] \\
&\leq  \mathbb{P}\left[ |\mathcal{P} - \mathbb{E}[\mathcal{P}]| \geq a_k\ell^{2k} \cdot \sqrt{\mathbb{E}_0[\mathcal{P}] \cdot \mathbb{E}_1[\mathcal{P}]} \right] \\
&\leq \exp\left[ -\ell + (2k - 1)n +O(1)\right] = o(1).
\end{align*}

Therefore, \( \mathcal{P} > 0 \) with high probability, and hence an \(\alpha\)-typical pair \( x, y \) is with high probability not an edge in \( G_p \). This completes the proof of Equation~\eqref{gadget} and thus of Theorem~\ref{thm: edge-density}.

\subsection{Proof of Theorem~\ref{thm: edge-expansion} (High Degree)}

We aim to show that for sufficiently large \( n \), with high probability,
\[
\deg(v) = (1 - o(1)) \cdot |V(G_p)| \qquad \text{for all } v \in V(G_p).
\]

Recall from Lemma~\ref{lem: russian gadget} that for \( x, y \in Q_p^n \), if \( |W(x, y)| \leq 1 \), then \( x \) and \( y \) form an edge in \( G_p \). Thus, it suffices to show that, with high probability, for every \( x \in Q_p^n \), the number of pairs \( (y, z) \in Q_p^n \times Q_p^n \) such that \( z \in W(x, y) \) is \( o(2^n p) = o(|Q_p^n|) \). This will immediately imply the desired degree bound.

We restrict attention to pairs \( (x, y) \) where \( y \) is \((x, \alpha)\)-typical for some small constant \( \alpha > 0 \), since by Lemma~\ref{lem: no too many atypical}, with high probability, each \( x \in Q_p^n \) has only \( o(|Q_p^n|) \) atypical neighbors. 
Fix \( x \in \mathcal{Q}^n \), and define a polynomial \( \mathcal{P} \) that counts the number of such witness pairs:
\[
\mathcal{P} := \sum X_y X_z,
\]
where the sum ranges over all \( y \in \mathcal{Q}^n \) that are \((x, \alpha)\)-typical, and all \( z \in W(x, y) \) such that the Hamming distance between \( z \) and \( x \) is at least \( \delta n \), for some fixed constant \( \delta > 0 \) (to be chosen later in terms of \( \varepsilon \)).

Each monomial \( X_y X_z \) is an indicator for the event that both \( y \) and \( z \) belong to \( Q_p^n \), and that \( z \in W(x, y) \). The lower bound on the Hamming distance is crucial to ensure a good upper bound on \( \mathbb{E}_1[\mathcal{P}] \) for applying the Kim–Vu concentration inequality. We now show that, with high probability, no pair of vertices in \( Q_p^n \) is closer than \( \delta n \), which justifies this restriction.

\paragraph{\textbf{Bounding Close Pairs.}}
Let \( \mathcal{E} \) be the event that there exists a pair \( x \neq y \in Q_p^n \) with Hamming distance at most \( \delta n \). Using the binary entropy function \( H(\delta) \),
\[
\mathbb{P}[\mathcal{E}] \leq 2^n \cdot \binom{n}{\delta n} \cdot p^2 \leq 2^{(1 + H(\delta)) n} \cdot (1 - \varepsilon)^{2n},
\]
which is \( o(1) \) for small enough \( \delta = \delta(\varepsilon) \). This justifies excluding such pairs from \( \mathcal{P} \).

We now establish the following two claims:
\begin{enumerate}
    \item \( \mathbb{E}[\mathcal{P}] = o(|Q_p^n|) \),
    \item \( \mathbb{P}[\mathcal{P} > \varepsilon \cdot |Q_p^n|] = 2^{-\omega(n)} \).
\end{enumerate}

A union bound over all \( x \in \mathcal{Q}^n \) (of size \( 2^n \)) then completes the proof.

\paragraph{{\bf Proof of (1): Bounding \( \mathbb{E}[\mathcal{P}] \).}}
There are at most \( 2^n \) choices for \( y \in \mathcal{Q}^n \) that are \((x,\alpha)\)-typical. For each such \( y \), any \( z \in W(x, y) \) must satisfy \( x(i) = z(i) \) for every \( i \) where \( x(i) = y(i) \). Since the number of such indices is at least \( n/2 - 2\alpha n \), the number of degrees of freedom in choosing \( z \) is at most \( n/2 + 2\alpha n \), so
\[
|W(x, y)| \leq 2^{n/2 + 2\alpha n}.
\]
Each pair \( (y, z) \) appears in \( \mathcal{P} \) with probability \( p^2 \). Therefore,
\[
\mathbb{E}[\mathcal{P}] \leq 2^n \cdot 2^{n/2 + 2\alpha n} \cdot p^2 = 2^{n/2 + 2\alpha n} \cdot (1 - \varepsilon)^{2n},
\]
which is \( o(2^n p) \) for small enough \( \alpha < \varepsilon/10 \).

\paragraph{{\bf Proof of (2): Concentration.}} We now seek to apply the Kim–Vu polynomial concentration inequality (Theorem~\ref{lem: concentration inequality}). Recall that \( \mathbb{E}_1[\mathcal{P}] \) is the maximum expected value of any first-order partial derivative. There are two types of partial derivative to consider: $y$ which are $(x,\alpha)$-typical and witnesses $z$ which are far from $x$ in Hamming distance. 

\textbf{Case 1.} Fix a typical \( y \), and consider \( \partial \mathcal{P} / \partial X_y \). The number of valid \( z \in W(x, y) \) with \( \operatorname{dist}(x, z) \geq \delta n \) is at most \( 2^{n/2 + 2\alpha n} \). Hence,
\[
\mathbb{E}\left[\frac{\partial \mathcal{P}}{\partial X_y}\right] \leq 2^{n/2 + 2\alpha n} \cdot p = 2^{2\alpha n} \cdot (1 - \varepsilon)^n = o(1).
\]

\textbf{Case 2.} Fix a vertex \( z \) with \( \operatorname{dist}(x, z) \geq \delta n \), and consider \( \partial \mathcal{P} / \partial X_z \). For \( z \in W(x, y) \), it must hold that \( x \land y \leq z \leq x \lor y \), implying that \( y \in [x \land z, x \lor z] \). Since \( z\) is far from \( x \), this interval is of size at most \( 2^{(1 - \delta) n} \). Therefore,
\[
\mathbb{E}\left[\frac{\partial \mathcal{P}}{\partial X_z}\right] \leq 2^{(1 - \delta) n} \cdot p = 2^{-\delta n} \cdot 2^n p.
\]

Taking the maximum over both cases gives:
\(
\mathbb{E}_1[\mathcal{P}] \leq 2^{-\delta n} \cdot 2^n p\), so 
\begin{align*}
    \sqrt{\mathbb{E}[\mathcal{P}] \cdot \mathbb{E}_1[\mathcal{P}]} &\leq 2^n \cdot p^{1.5} \cdot 2^{n/4 + \alpha n - \delta n/2} \\
    & \leq 2^n\cdot  p \cdot (1-\varepsilon)^{n/2} \cdot 2^{-n/4} \cdot 2^{n/4 + \alpha n - \delta n/2}\\
    & \leq    2^{-\delta n / 2} \cdot 2^n  \cdot p, 
\end{align*}
if  $\alpha$ is small enough with respect to $\ep.$

Applying Theorem~\ref{lem: concentration inequality} with, say, \( \ell = (1 + \delta/2)^{n/3} \), we obtain
\[
\mathbb{P}\left[\left| \mathcal{P} - \mathbb{E}[\mathcal{P}] \right| \geq \varepsilon \cdot 2^n p \right] \leq \exp\left(-\ell + \log(2^n) + O(1)\right) = 2^{-\omega(n)}.
\]

\paragraph{\textbf{Final Union Bound.}}
The number of vertices in \( Q_p^n \) is \( |Q_p^n| \leq 2^n p \). Combining the tail bound and the estimate on \( \mathcal{E} \), we conclude:
\[
\mathbb{P}\left[\exists x \in Q_p^n \text{ with } \deg(x) \geq \varepsilon \cdot 2^n p\right] \leq 2^n p \cdot 2^{-\omega(n)} + o(1) = o(1).
\]

\textit{Remark.} The tail bound from Kim–Vu is doubly exponential in \( n \), while a union bound adds only an exponential factor. Thus it is straightforward to control the global deviation over all vertices.

\subsection{Proof of Theorem \ref{thm: clique threshold} (1) (Clique)}

We begin by showing that for \( p \leq 2^{-0.8305n} \), the graph \( G_p \) is, with high probability, a clique. This is already an improvement of the current best bound. We then refine this estimate to obtain the sharp threshold.

By Lemma~\ref{lem: russian gadget} (2), it suffices to establish that for all \( x, y \in Q_p^n \), we have \( |W(x,y)| \leq 1 \).
For \( \delta \in [0,1] \), define \( \mathcal{E}_{\delta} \) as the set of tuples \( (x, y, u, v) \) such that the Hamming distance between \( x \) and \( y \) is exactly \( \delta n \), and \( u, v \in W(x,y) \).
There are \( 2^n \) choices for \( x \), and given \( x \), the number of points \( y \) at distance exactly \( \delta n \) from \( x \) is at most
\[
\binom{n}{\delta n} \leq 2^{H(\delta)n},
\]
where \( H(\delta) = -\delta \log_2 \delta - (1-\delta) \log_2 (1-\delta) \) is the binary entropy function. Since \( x \) and \( y \) differ on exactly \( \delta n \) coordinates, the number of possible pairs \( (u,v) \) with \( x \land y \leq u,v \leq x \lor y \) is \( 2^{2\delta n} \). Hence,
\[
|\mathcal{E}_{\delta}| \leq 2^{(1 + 2\delta + H(\delta))n}.
\]
Let \( f(\delta) := 1 + 2\delta + H(\delta) \).
Standard calculus techniques or computational tools show that \( f(\delta) \) is maximized at \( \delta = 4/5 \), with maximum value
\[
f\left(\tfrac{4}{5}\right)  \approx 1 + 1.6 + 0.7219 = 3.3219.
\]
Since each tuple \( (x,y,u,v) \in \mathcal{E}_{\delta} \) appears with probability \( p^4 \), the expected number of such tuples is at most
\[
\int_0^1 2^{(1 + H(\delta) + 2\delta)n} \cdot p^4 \, d\delta \leq 2^{3.3219 n} \cdot p^4.
\]
This tends to zero as \( n \to \infty \), provided that
\[
p^4 \leq 2^{-3.3219 n} \quad \text{or equivalently} \quad p \leq 2^{-3.3219n/4} \approx 2^{-0.8305 n}.
\]

Therefore, under the condition \( p \leq 2^{-0.8305 n} \), with high probability we have \( |W(x,y)| \leq 1 \) for all \( x,y \in Q_p^n \), which implies \( G_p \) is a clique.

To obtain a sharper bound on \( p \), we condition on the event that there are no pairs \( x, y \in Q_p^n \) with Hamming distance greater than some fixed \( \delta \in [0,1] \). In particular, we choose \( \delta>1/2 \) to be the smallest value such that
\[
2^{(1 + H(\delta))n} \cdot p^2 = o(1),
\]
which is equivalent to requiring
\(
p \leq 2^{-(1 + H(\delta))n/2}.
\)

Under this conditioning, it suffices to ensure that for all pairs \( x,y \in Q_p^n \) with \( \dist(x,y) \leq \delta n \),  \( |W(x,y)| \leq 1 \). The probability that there exists pairs at a greater distance is negligible. The number of such pairs is at most \( 2^{(1 + H(\delta))n} \), and for each such pair, the number of valid witness tuples \( (u,v) \) remains bounded by \( 2^{2\delta n} \). Thus, the expected number of problematic quadruples is bounded by
\[
2^{(1 + H(\delta) + 2\delta)n} \cdot p^4.
\]
We observe that
\[
\frac{1 + H(\delta) + 2\delta}{4} = \frac{1 + H(\delta)}{2},
\]
 holds precisely when \(\delta = (1 + H(\delta))/2 \). Solving this equation yields \( \delta \approx 0.8295 \), and therefore the slightly improved condition becomes
\[
p \leq 2^{-(1 + H(\delta))n/2} =2^{-\delta n} \approx 2^{-0.8295 n}.
\]

\subsection{Proof of Theorem \ref{thm: clique threshold} (2) (Not A Clique)}

Let \( \delta = (1 + H(\delta))/2 \approx 0.8295 \)  be as in the proof of Theorem \ref{thm: clique threshold} part (1). 
We show that for every \( \varepsilon > 0 \), if 
\[
p \geq  (1 + \varepsilon)^n \cdot 2^{-\delta n},
\]
then with high probability the graph \( G_p \) is not a clique.

Let \( k \) be sufficiently large. We will prove, using the Kim–Vu concentration inequality, that with high probability there exist \( x, y, u_1, \ldots, u_{2k} \in Q_p^n \) such that the Hamming distance between \( x \) and \( y \) is \( \delta n \), and
\[
\frac{1}{2}(x + y) = \frac{1}{2k} \sum_{i=1}^{2k} u_i.
\]
Then, by Lemma~\ref{lem: russian gadget}, it follows that \( xy \notin E(G_p) \), and hence \( G_p \) is not a clique.
To this end, define the polynomial
\[
\mathcal{P} := \sum X_x X_y X_{u_1} \cdots X_{u_{2k}},
\]
where the sum is over all tuples \( x, y, u_1, \ldots, u_{2k} \) satisfying the above condition. If $\mathcal{P}$ is not zero, then there is some pair $x,y \in V(G_p)$ which does not form an edge.

\paragraph{{\bf Lower bound on \( \mathbb{E}[\mathcal{P}] \).}}
We now estimate the expectation of \( \mathcal{P} \). There are at least \( 2^{(1+ H(\delta))n} /(n+1) \) choices for the pair \( (x, y) \) with Hamming distance \( \delta n \). For each such pair, we count the tuples \( u_1, \ldots, u_{2k} \in \{0,1\}^n \) that average to \((x + y)/2\).

In each of the \( \delta n \) coordinates where \( x \) and \( y \) differ,  \( (x+y)/2 \) takes the value $1/2$. So we require that exactly half of the \( u_i \)'s have a 1 and the other half have a 0 in that coordinate. This can be done in \( \binom{2k}{k} \) ways for each such coordinate. In the remaining \( (1 - \delta)n \) coordinates where \( x \) and \( y \) agree, all \( u_i \)'s must match \( x \) and \( y \) in those positions, so there is only one valid assignment. Hence, the total number of such tuples of $u$'s is  \( \binom{2k}{k}^{\delta n} \).

Each configuration $ x,y, u_1, \ldots u_{2k}$ contributes \( p^{2k+2} \) to the expectation, since each point appears independently with probability \( p \). Therefore,
\[
\mathbb{E}[\mathcal{P}] \geq  \frac{1}{n+1}\cdot 2^{(1 + H(\delta))n} \cdot \binom{2k}{k}^{\delta n} \cdot p^{2k+2}.
\]

Using the standard estimate \( \binom{2k}{k} \sim \frac{4^k}{\sqrt{\pi k}} \), we obtain
\[
\mathbb{E}[\mathcal{P}]  \geq  \frac{1}{n+1} \cdot 2^{(1 + H(\delta) + 2k\delta)n} \cdot \Theta(k)^{-\delta n/2} \cdot p^{2k+2}.
\]

Substituting \( p = (1 + \varepsilon)^n  2^{-\delta n} \), using $\delta = (1 + H(\delta))/2$, and simplifying yields
\[
\mathbb{E}[\mathcal{P}]  \geq \frac{1}{n+1} \cdot \Theta(k)^{-\delta n/2} \cdot (1 + \varepsilon)^{(2k + 2)n} .
\]

Finally, since \( \varepsilon > 0 \) is fixed, we can choose \( k \) sufficiently large (e.g., depending on \( 1/\varepsilon \)) so that the exponential growth from the \( (1 + \varepsilon)^{(2k+2)n} \) term dominates the decay in the factor $ \frac{1}{n+1} \Theta(k)^{-\delta n/2}$. Therefore, we conclude that
\[
\mathbb{E}[\mathcal{P}] \to \infty \quad \text{as } n \to \infty, \text{ as claimed.}
\]
\paragraph{{\bf Upper bound on \( \mathbb{E}_1[\mathcal{P}] \).}}
We now turn to the second moment. Recall that
\[
\mathbb{E}_1[\mathcal{P}] := \max_{v \in \mathcal{Q}^n} \mathbb{E}\left[\frac{\partial \mathcal{P}}{\partial X_v}\right],
\]
 We will show that
\[
\mathbb{E}_1[\mathcal{P}] \leq \frac{\mathbb{E}[\mathcal{P}]}{(1 + \beta)^n}
\]
for some \( \beta > 0 \) depending on \( \varepsilon \).
 We need only to consider partial derivatives with respect to the variables \(X_x, X_y, \) and \(X_{u_i}\).
 The partial derivative \( \frac{\partial \mathcal{P}}{\partial X_x} \) is a degree \( (2k+1) \) polynomial that counts the number of tuples \( (y, u_1, \ldots, u_{2k}) \) such that:
\begin{itemize}
    \item The Hamming distance between $x$ and $y$ is \(\delta n \), and
    \item \( \frac{1}{2}(x + y) = \frac{1}{2k} \sum u_i \).
\end{itemize}

For each such \( y \), there are \( \binom{2k}{k}^{\delta n} \) such tuples of \( u_i \)'s, and there are at most \( 2^{H(\delta)n} \) such \( y \)'s. Thus, we get:
\[
\mathbb{E}\left[\frac{\partial \mathcal{P}}{\partial X_x}\right] \leq 2^{H(\delta)n} \cdot \binom{2k}{k}^{\delta n} \cdot p^{2k+1}.
\]

Recalling that
\[
\mathbb{E}[\mathcal{P}] \geq \frac{1}{n+1}\cdot 2^{(1 + H(\delta))n} \cdot \binom{2k}{k}^{\delta n} \cdot p^{2k+2},
\]
we obtain:
\[
\mathbb{E}\left[\frac{\partial \mathcal{P}}{\partial X_x}\right] \leq \frac{\mathbb{E}[\mathcal{P}]}{2^n p} \cdot (n+1) \leq  \frac{\mathbb{E}[\mathcal{P}]}{(1+\varepsilon)^n}  \cdot\frac{(n+1)}{2^{(1-\delta)n}}   .
\]

Recalling that $\delta \approx .8295$, as $n$ grows large, the second term $ (n+1) 2^{(\delta-1)n}$ is less than one (in fact, it suffices to take $n > 29$). Therefore, we have the desired bound with $\beta = \ep$:
\[
\mathbb{E}\left[\frac{\partial \mathcal{P}}{\partial X_x}\right] \leq \frac{\mathbb{E}[\mathcal{P}]}{(1+\varepsilon)^n}.
\]

 The case of  \( X_y \) is the same as  \( X_x \) by symmetry.

Lastly, fix \( u \in \mathcal{Q}^n \) corresponding to one of the witness vertices \( u_i \). The partial derivative \( \frac{\partial \mathcal{P}}{\partial X_u} \) counts the number of tuples \( (x, y, u_1, \ldots, \hat{u}, \ldots, u_{2k}) \) such that \( u \) appears in the multiset \( \{u_1, \ldots, u_{2k}\} \), and the midpoint condition holds.

We count how many configurations like this exist. Consider a fixed pair \( (x, y) \) of Hamming distance \( \delta n \), and fix a witness \( u  \in \{u_1, \ldots, u_{2k}\}\). There are  \( \delta n \) coordinates where $x(i) \neq y(i).$ If $u(i) = 1$,
then we must choose \( k-1 \) of the remaining \( 2k-1 \) u's to be 1's in coordinate $i$, noting that $\binom{2k-1}{k-1} = \binom{2k-1}{k}$. If $u(i) = 0$, then we must choose \( k \) of the remaining \( 2k-1 \) u's to be 1's in coordinate $i$. If $x(i) = y(i)$ then there is no choice for coordinate $i$ of any witness $u$. Thus the number of valid ways to choose the remaining \( 2k-1 \) $u_i$'s is \(\binom{2k-1}{k}^{\delta n}. \)
Hence,
\[
\mathbb{E}\left[\frac{\partial \mathcal{P}}{\partial X_u}\right] \leq 2^n \cdot 2^{H(\delta)n} \cdot \binom{2k - 1}{k}^{\delta n} \cdot p^{2k + 1}.
\]

Using that \( \binom{2k - 1}{k} = \binom{2k}{k} \cdot 2^{-1} \), we have:
\[
\binom{2k - 1}{k}^{\delta n} =2^{-\delta n} \cdot \binom{2k}{k}^{\delta n}. 
\]
Thus we get the following bound in terms of $\mathbb{E}[\mathcal{P}]$. 
\[
\mathbb{E}\left[\frac{\partial \mathcal{P}}{\partial X_u}\right] \leq  \frac{\mathbb{E}[\mathcal{P}] (n+1) }{p 2^{\delta n} }
\]
Since \(p \geq  (1 + \varepsilon)^n \cdot 2^{-\delta n},\) we see that
\[
\mathbb{E}\left[\frac{\partial \mathcal{P}}{\partial X_u}\right] \leq \mathbb{E}[\mathcal{P}] \cdot (1 + \varepsilon)^{n}.
\]

Taking $\beta = \min \{ \ep, (1+\ep)^{-1} - 1\}$ (which we note is positive), shows that 
\[
\mathbb{E}_1[\mathcal{P}] \leq \frac{\mathbb{E}[\mathcal{P}]}{(1 + \beta)^n}.
\]
\paragraph{\bf{Concentration}}
We are now in position to apply the Kim–Vu concentration inequality (see Lemma~\ref{lem: concentration inequality}). By our calculations, 
\[ \sqrt{\mathbb{E}[\mathcal{P}] \cdot \mathbb{E}_1[\mathcal{P}]}  \leq \mathbb{E}[\mathcal{P}] (1+\beta)^{-n/2} . \]

So, 
\begin{align*}
\mathbb{P}[\mathcal{P} = 0] 
&\leq \mathbb{P}\left[ |\mathcal{P} - \mathbb{E}[\mathcal{P}]| \geq \mathbb{E}[\mathcal{P}] \right] \\
&\leq \mathbb{P}\left[ |\mathcal{P} - \mathbb{E}[\mathcal{P}]| \geq \sqrt{\mathbb{E}[\mathcal{P}] \cdot \mathbb{E}_1[\mathcal{P}]}   (1+\beta)^{n/2}\right].
\end{align*}

Taking (say) \( \ell := (1 + \beta)^{n / (5k)} \), we obtain
\begin{align*}
\mathbb{P}[\mathcal{P} = 0] &\leq \exp\left[ -\ell + (2k + 1)n +O(1) \right] = o(1).
\end{align*}
Therefore, with high probability, \( \mathcal{P} > 0 \) as desired.

\printbibliography

@article{LerouxRademacherExpansion,
author = {Leroux, Brett and Rademacher, Luis},
title = {Expansion of random 0/1 polytopes},
journal = {Random Structures \& Algorithms},
volume = {64},
number = {2},
pages = {309-319},
year = {2024}
}

@article{BondarenkoBrodskiiCliques,
title = {On random 2-adjacent 0/1-polyhedra},
author = {V. A. Bondarenko and A. G. Brodskii},
pages = {181 -- 186},
volume = {18},
number = {2},
journal = {Discrete Mathematics and Applications},
year = {2008}
}

@InProceedings{KaibelRemshagenDensity,
author="Kaibel, Volker
and Remshagen, Anja",
editor="Arora, Sanjeev
and Jansen, Klaus
and Rolim, Jos{\'e} D. P.
and Sahai, Amit",
title="On the graph-density of random 0/1-polytopes",
booktitle="Approximation, Randomization, and Combinatorial Optimization.. Algorithms and Techniques",
year="2003",
publisher="Springer Berlin Heidelberg",
address="Berlin, Heidelberg",
pages="318--328"
}

@InProceedings{MihailVaziraniConj,
author="Mihail, Milena",
editor="Havel, Ivan M.
and Koubek, V{\'a}clav",
title="On the expansion of combinatorial polytopes",
booktitle="Mathematical Foundations of Computer Science 1992",
year="1992",
publisher="Springer Berlin Heidelberg",
address="Berlin, Heidelberg",
pages="37--49"
}

@article{KimVuConcentration,
title = {Concentration of multivariate polynomials and its applications},
author = {J. Kim and V. Vu},
pages = {417 - 434},
volume = {20},
journal = {Combinatorica},
year = {2000}
}

@book{JLR,
  title={Random Graphs},
  author={Janson, S.  and Luczak, T. and Ruci\'{n}ski, A.},
  year={2000},
  publisher={John Wiley \& Sons}
}

@book {ZieglerBook,
    AUTHOR = {Ziegler, G\"{u}nter M.},
     TITLE = {Lectures on Polytopes},
    SERIES = {Graduate Texts in Mathematics},
    VOLUME = {152},
 PUBLISHER = {Springer-Verlag, New York},
      YEAR = {1995}
}

@book {GrunbaumBook,
    AUTHOR = {Gr\"{u}nbaum, Branko},
     TITLE = {Convex Polytopes},
    SERIES = {Graduate Texts in Mathematics},
    VOLUME = {221},
   EDITION = {Second},
 PUBLISHER = {Springer-Verlag, New York},
YEAR = {2003}
}

@book {SchrijverBooks,
    AUTHOR = {Alexander Schrijver},
     TITLE = {Combinatorial Optimization. Polyhedra and Efficiency},
    SERIES = {Volume 24 of Algorithms and Combinatorics},
    VOLUME = {A-C},
 PUBLISHER = {Springer-Verlag, Berlin},
YEAR = {2003}
}

@article{logconcavepolynomials,
    author = "Nima Anari and Kuikui Liu and Shayan Oveis Gharan and Cynthia Vinzant",
    title = "Log-concave polynomials II: High-dimensional walks and an FPRAS for counting bases of a matroid" ,
    journal = "Ann. of Math. (2) ",
    volume = "199(1)",
    year = "2024",
    pages = "259 -299"
}

\end{document}